\documentclass{article}
\usepackage{amsmath, amsthm, amsfonts, amssymb}
\usepackage{graphicx}

\def\chitilde{\tilde{\chi}}
\def\Normal{{\mathcal N}}
\def\N{{\mathbb N}}
\def\Z{{\mathbb Z}}
\def\R{{\mathbb R}}
\def\C{{\mathbb C}}

\def\E{{\mathbb E}}

\def\betah{{\hat\beta}}
\def\F{{\mathcal F}}

\DeclareMathOperator{\Tr}{Tr}

\newtheorem{thm}{Theorem}

\newtheorem{lem}[thm]{Lemma}
\newtheorem{prop}[thm]{Proposition}

\theoremstyle{definition}

\theoremstyle{remark}
\newtheorem{remark}[thm]{Remark}

\long\def\symbolfootnote[#1]#2{\begingroup%
\def\thefootnote{\fnsymbol{footnote}}\footnote[#1]{#2}\endgroup}

\title{The mean spectral measures of random Jacobi matrices related to Gaussian beta ensembles}
\author{Trinh Khanh Duy and Tomoyuki Shirai}
\begin{document}
\maketitle
\begin{abstract}
	An explicit formula for the mean spectral measure of a random Jacobi matrix is derived. The matrix may be regarded as the limit of Gaussian beta ensemble (G$\beta$E) matrices as the matrix size $N$ tends to infinity with the constraint that $N \beta $ is a constant.  
\end{abstract}

{\bf Keywords.} random Jacobi matrix, Gaussian beta ensemble, spectral measure, self-convolutive recurrence

{{2010 Mathematics Subject Classification. } Primary 47B80; secondary 15A52, 44A60, 47B36}


\section{Introduction}

The paper studies spectral measures of random (symmetric) Jacobi matrices of the form
\[
	J_{\alpha} = \begin{pmatrix}
		\Normal(0,1)		&\chitilde_{2\alpha}	\\
		\chitilde_{2\alpha}	&\Normal(0,1)		&\chitilde_{2\alpha}	\\
							&\ddots		&\ddots		&\ddots
	\end{pmatrix},\quad (\alpha > 0),
\]
where the diagonal is an i.i.d.~(independent identically distributed) sequence of standard Gaussian $\Normal(0,1)$ random variables, the off diagonal is also an i.i.d.~sequence of $\chitilde_{2\alpha}$-distributed random variables. Here $\chitilde_{2\alpha}=\chi_{2\alpha}/\sqrt{2}$ with $\chi_{2\alpha}$ denoting the chi distribution with $2\alpha$ degree of freedom. As explained later, $J_\alpha$ is regarded as the limit of Gaussian beta ensembles (G$\beta$E for short) as the matrix size $N$ tends to infinity and the parameter $\beta$ also varies with the constraint that $N\beta = 2\alpha$.

Let us explain some terminologies and introduce main results of the paper. A (semi-infinite) Jacobi matrix is a symmetric tridiagonal matrix of the form
\[
				J = \begin{pmatrix}
					a_1		&b_1		\\
					b_1		&a_2		&b_2\\
							&\ddots		&\ddots		&\ddots	\\
				\end{pmatrix}, \text{ where } a_i \in \R, b_i > 0.
\]
For a Jacobi matrix $J$, there is a probability measure $\mu$ on $\R$ such that 
\[
	\int_\R x^k d\mu  = \langle J^k e_1, e_1\rangle =J^k(1,1), \quad k = 0,1,\dots,
\]
where $e_1 = (1,0,\dots)^T \in \ell^2$. Here $\langle u, v \rangle$ denotes the inner product of $u$ and $v$ in $\ell^2$, while $\langle \mu, f \rangle := \int f d\mu$ will be used to denote the integral of a function $f$ with respect to a measure $\mu$. Then the measure $\mu$ is unique if and only if $J$, as a symmetric operator defined on $D_0 = \{x=(x_1,x_2,\dots.): x_k = 0 \text{ for $k$ sufficiently large}\}$,  is essentially self-adjoint, that is, $J$ has a unique self-adjoint extension in $\ell^2$. When the measure $\mu$ is unique, it is  called the spectral measure of $J$, or more precisely, the spectral measure of $(J, e_1)$. It is known that the condition 
\[
	\sum_{i = 1}^\infty \frac{1}{b_i} = \infty
\]
implies the essential self-adjointness of $J$, \cite[Corollary 3.8.9]{Simon2011}.

For the random Jacobi matrix $J_\alpha$, the above condition holds almost surely  because its off diagonal elements are positive i.i.d.~random variables. Thus spectral measures $\mu_\alpha$ are uniquely determined by the following relations  
\[
	\langle \mu_\alpha, x^k\rangle= J_\alpha^k(1,1), \quad k = 0,1,\dots.
\]
Then the mean spectral measure $\bar\mu_\alpha$ is defined to be a probability measure satisfying 
\[
	\langle \bar\mu_\alpha, f\rangle = \E[\langle \mu_\alpha, f\rangle], 
\] 
for all bounded continuous functions $f$ on $\R$. It then follows that 
\[
				\langle \bar \mu_\alpha, x^k\rangle = \E[\langle \mu_\alpha, x^k\rangle ], \quad k = 0,1,\dots,
\]
provided that the right hand side of the above equation is finite for all $k$.

The purpose of this paper is to identify the mean spectral measure $\bar \mu_\alpha$. Our main results are as follows.
\begin{thm}\label{thm:explicit-formula-for-mu-bar}
\begin{itemize}
\item[\rm (i)] The mean spectral measure $\bar \mu_\alpha$ coincides with the spectral measure of the non-random Jacobi matrix $A_\alpha$, where
\[
	A_\alpha = \begin{pmatrix}
		0		&\sqrt{\alpha + 1}	\\
		\sqrt{\alpha + 1}	&0		&\sqrt{\alpha + 2}		\\
												&\ddots		&\ddots		&\ddots
	\end{pmatrix}.
\]
\item[\rm (ii)]
The measure $\bar \mu_\alpha$ has the following density function
\[
		\bar \mu_\alpha (y) = \frac{e^{-y^2/2}}{\sqrt{2\pi}}\frac{1}{|\hat f_\alpha(y)|^2} ,
\]
where 
\[
	\hat f_\alpha(y) =\sqrt{\frac 2\pi} \int_0^\infty f_\alpha(t) e^{iyt}dt, \quad f_\alpha(t) = \pi \sqrt{\frac{\alpha}{\Gamma(\alpha)}} t^{\alpha - 1} \frac{e^{-\frac {t^2}{2}}}{\sqrt{2\pi}}.
\]
\end{itemize}
\end{thm}

Let us sketch out main ideas for the proof of the above theorem. To show the first statement, the key idea is to regard the Jacobi matrix $J_\alpha$ as the limit of G$\beta$E as the matrix size $N$ tends to infinity with $N\beta = 2\alpha$. More specifically, let $T_N(\beta)$ be a finite random Jacobi matrix whose components are (up to the symmetry constraints) independent and are distributed as
\[
	T_N(\beta) = \begin{pmatrix}
		\Normal(0,1)		&\chitilde_{(N-1)\beta}	\\
		\chitilde_{(N-1)\beta}	&\Normal(0,1)		&\chitilde_{(N-2)\beta}		\\
					
												&\ddots		&\ddots		&\ddots \\
						
						&&					\chitilde_{\beta}	&\Normal(0,1)
	\end{pmatrix}.
\]
Then it is well known in random matrix theory that the eigenvalues of $T_N(\beta)$ are distributed as G$\beta$E, namely, 
\begin{equation*}
	(\lambda_1, \dots, \lambda_N) \propto  \prod_{l = 1}^N e^{-\lambda_l^2 /2} \prod_{1 \le j < k \le N} |\lambda_k - \lambda_j|^\beta.
\end{equation*}
Moreover, by letting $N \to \infty$ with $\beta = 2\alpha/N$, the matrices $T_N(\beta)$ converge, in some sense, to $J_\alpha$. 
That crucial observation together with a result on moments of G$\beta$E (\cite[Theorem~2.8]{DE06}) makes it possible to show that $\bar \mu_\alpha$ coincides with the spectral measure of $A_\alpha$.

The next step is to establish the following self-convolutive recurrence for even moments of $\bar \mu_\alpha$, 
\[
	u_n(\alpha) = (2n -1)u_{n-1}(\alpha) + \alpha \sum_{i = 0}^{n - 1} u_i(\alpha) u_{n - 1 - i}(\alpha),
\]
where $u_n(\alpha)$ is the $2n$th moment of $\bar \mu_\alpha$. Note that its odd moments are all vanishing because the spectral measure of $A_\alpha$ is symmetric. Finally, the explicit formula for $\bar \mu_\alpha$ is derived by using the method in \cite{Martin2010}. 

The paper is organized as follows. In the next section, we mention some known results on G$\beta$E needed in this paper. In Section~3, we introduce the matrix model and step by step, prove the main theorem.

\section{A result on Gaussian $\beta$-ensembles}
The Jacobi matrix model for G$\beta$E, a finite random Jacobi matrix, was discovered by Dumitriu and Edelman \cite{DE02}. First of all, let us mention some preliminary facts about finite Jacobi matrices. Assume that $J$ is a finite Jacobi matrix of order $N$ (with the requirement that the off diagonal elements are positive). Then the matrix $J$ has exactly $N$ distinct eigenvalues $\lambda_1, \lambda_2, \dots, \lambda_N$. Let $v_1, v_2, \dots, v_N$ be the corresponding eigenvectors which are chosen to be an orthonormal basis in $\R^N$. Then the spectral measure $\mu$, which is well defined by $\langle \mu, x^k \rangle = J^k(1,1), k=0,1,\dots,$ can be  expressed as
\[
	\mu = \sum_{j = 1}^N q_j^2 \delta_{\lambda_j}, \quad q_j = |v_j(1)|,
\]
where $\delta_\lambda$ denotes the Dirac measure. It is known that a finite Jacobi matrix of order $N$ is one-to-one correspondence with a probability measure supported on $N$ points, or a set of Jacobi matrix parameters $\{a_i\}_{i = 1}^N, \{b_j\}_{j = 1}^{N - 1}$ is one-to-one correspondence with the spectral data $\{\lambda_i\}_{i = 1}^N, \{q_j\}_{j = 1}^N$.

The Jacobi matrix model for G$\beta$E is defined as follows. Let $\{a_i\}_{i = 1}^N$ be an i.i.d.~sequence of standard Gaussian  $\Normal(0,1)$ random variables and $\{b_j\}_{j = 1}^{N - 1}$ be a sequence of independent random variables having $\tilde \chi$ distributions with parameters $(N-1)\beta, (N-2)\beta, \dots, 1$, respectively, which is independent of $\{a_i\}_{i = 1}^N$. Here $\tilde \chi_k,$ for $k > 0$, denotes the distribution with the following probability density function 
\[
	\frac{2}{\Gamma(k/2)} u^{k - 1} e^{-u^2}, u >0,
\]
which is nothing but $\chi_k/\sqrt{2}$, or the square root of the gamma distribution with parameter $(k/2,1)$. 
We form a random Jacobi matrix $T_N(\beta)$ from $\{a_i\}_{i = 1}^N$ and $\{b_j\}_{j = 1}^{N - 1}$ as follows,    
\[
	T_N(\beta) = \begin{pmatrix}
		\Normal(0,1)		&\chitilde_{(N-1)\beta}	\\
		\chitilde_{(N-1)\beta}	&\Normal(0,1)		&\chitilde_{(N-2)\beta}		\\
					
												&\ddots		&\ddots		&\ddots \\
						
						&&					\chitilde_{\beta}	&\Normal(0,1)
	\end{pmatrix}.
\]
Then the eigenvalues $\{\lambda_i\}_{i = 1}^N$ and the weights $\{q_j\}_{j = 1}^N$ are independent, with the distribution of the former given by
\[
	(\lambda_1,\lambda_2, \dots, \lambda_N) \propto  \prod_{l = 1}^N e^{-\lambda_l^2 /2} \prod_{1 \le j < k \le N} |\lambda_k - \lambda_j|^\beta, 
\]
and the distribution of the latter given by
\[
	(q_1, q_2, \dots, q_N) \propto \frac{1}{ q_N} \prod_{i = 1}^N q_i^{\beta - 1}, \quad (q_i > 0, \sum_{i = 1}^N q_i^2 = 1).
\]
It is also known that $q = (q_1, \dots, q_N)$ is distributed as a vector $(\chitilde_{\beta}, \dots, \chitilde_{\beta})$ with i.i.d.~components, normalized to unit length.

The trace of $T_N(\beta)^n$ and $T_N(\beta)^n (1, 1)$ can be expressed in term of the spectral data as
\[
	\Tr (T_N(\beta)^n) = \sum_{j = 1}^N \lambda_j ^n, \quad T_N(\beta)^n (1, 1) = \sum_{j = 1}^N q_j^2 \lambda_j ^n.
\]
Consequently,
\begin{align*}
	\E[T_N(\beta)^n (1, 1)] &= \E[\sum_{j = 1}^N q_j^2 \lambda_j ^n] = \sum_{j = 1}^N \E[q_j^2] \E[\lambda_j^n] = \frac{1}{N}\sum_{j = 1}^N \E[\lambda_j^n] \\
	&= \frac{1}{N} \E[\Tr (X_N(\beta)^n)].
\end{align*}

In the rest of this section, for convenience, we use the parameter $\betah = \beta/2$. Let $m_{p}(N,\betah) = \E[T_N(2\betah)^{2p} (1, 1)]$. It is clear that $m_{p}(N, \betah)$ is a polynomial of degree $p$ in $N$, and thus $m_{p}(N, \betah)$ is defined for all $N \in \R$. 
Then a result for the trace of $T_N(\beta)^n$ can be rewritten for $m_{p}(N,\betah)$ as follows.
\begin{thm}[cf. {\cite[Theorem~2.8]{DE06}} and {\cite[Theorem~2]{Forrester2013}}]\label{thm:Forrester2013-Thm2}
It holds that
\[
	m_{p}(N , \betah) = (-1)^p \betah^p m_{p}(-\betah N, \betah^{-1}).
\]
\end{thm}

Observe that $\betah^{-p} m_{p}(N, \betah)$ is the expectation of the $2p$th moment of the spectral measure of the following Jacobi matrix
\[
	\frac{1}{\sqrt{\betah}}T_N(2\betah) = \frac{1}{\sqrt{\betah}}\begin{pmatrix}
		\Normal(0,1)		&\chitilde_{(N-1)2\betah}	\\
		\chitilde_{(N-1)2\betah}	&\Normal(0,1)		&\chitilde_{(N-2)2\betah}		\\
						
												&\ddots		&\ddots		&\ddots \\
						&&					\chitilde_{2\betah}	&\Normal(0,1)
	\end{pmatrix}.
\]
As $\betah \to \infty$, it holds that 
\[
	\frac{\Normal(0,1)}{\sqrt{\betah}} \to 0, \quad \frac{\chitilde_{k 2\betah}}{\sqrt{\betah}} = \left( \frac{\Gamma(k\betah, 1)}{\betah}\right)^{1/2}  \to \sqrt{k} \text{ (in $L^q$ for any $q \ge 1$).}
\]
The convergences also hold almost surely. Therefore as $\betah \to \infty$,
\[
	\frac{1}{\sqrt{\betah}}T_N(2\betah) \to \begin{pmatrix}
		0		& \sqrt{N - 1}\\
		\sqrt{N - 1}	&0		&\sqrt{N - 2}		\\
												&\ddots		&\ddots		&\ddots \\
						&&					1&0
	\end{pmatrix} =: H_N.
\]
Here the convergence of matrices means the convergence (in $L^q$) of their elements.
Let $h_{p}(N)= H_N^{2p}(1,1) $ for  $N >p$. Then $h_{p}(N)$ is a polynomial of degree $p$ in $N$ so that $h_{p}(N)$ is defined for all $N \in \R$. The above convergence of matrices implies that for fixed $p$ and fixed $N$,
\begin{equation}\label{kappa-to-infty}
	h_{p}(N) = \lim_{\betah \to \infty} \betah^{-p} m_{p}(N, \betah).
\end{equation}

Let 
\[
	A_\alpha = \begin{pmatrix}
		0		&\sqrt{\alpha + 1}	\\
		\sqrt{\alpha + 1}	&0		&\sqrt{\alpha + 2}		\\
											&\ddots		&\ddots		&\ddots
	\end{pmatrix},
\]
and let $u_{p}(\alpha) = A_\alpha^{2p}(1,1)$. Then $u_{p}(\alpha)$ is also a polynomial of degree $p$ in $\alpha$. In addition, it is easy to see that
\begin{equation}\label{relation-of-u-and-h}
	u_{p}(\alpha) = (-1)^p h_{p}(-\alpha).
\end{equation}
As a direct consequence of Theorem~\ref{thm:Forrester2013-Thm2} and relations \eqref{kappa-to-infty} and \eqref{relation-of-u-and-h}, we get the following result.
\begin{prop}\label{thm:limit-of-m}
	As $N \to \infty$ with $\betah = \betah(N) = \alpha/N$,
	\[
		m_{p}(N, \betah) \to u_{p}(\alpha) =A^{2p}_\alpha(1,1).
	\]
\end{prop}

\section{Random Jacobi matrices related to Gaussian $\beta$ ensembles}
\subsection{A matrix model and proof of Theorem~\ref{thm:explicit-formula-for-mu-bar}(i)}
Consider the following random Jacobi matrix
\[
	J_{\alpha} = \begin{pmatrix}
		\Normal(0,1)		&\chitilde_{2\alpha}	\\
		\chitilde_{2\alpha}	&\Normal(0,1)		&\chitilde_{2\alpha}	\\
										&\ddots		&\ddots		&\ddots
	\end{pmatrix},
\]
where all components are independent random variables. More precisely, the diagonal $\{a_i\}_{i = 1}^\infty$ is an i.i.d.~sequence of standard Gaussian  $\Normal(0,1)$ random variables and the off diagonal $\{b_j\}_{j = 1}^\infty$ is another i.i.d.~sequence of $\chitilde_{2 \alpha}$ random variables. Then the spectral measure $\mu_\alpha$ of $J_{\alpha}$ exists and is unique almost surely because 
\[
	\sum_{j = 1}^\infty \frac{1}{b_j} = \infty \text{(almost surely).}
\]

The mean spectral measure $\bar \mu_\alpha$ is defined to be a probability measure satisfying
\[
	\langle \bar \mu_\alpha, f\rangle = \E[\langle \mu, f\rangle],  
\]
for all bounded continuous functions $f$ on $\R$. Then Theorem~{\rm\ref{thm:explicit-formula-for-mu-bar}(i)} states that the measure $\bar \mu_\alpha$ coincides with the spectral measure of $(A_\alpha, e_1)$.
\begin{proof}[Proof of Theorem~{\rm\ref{thm:explicit-formula-for-mu-bar}(i)}]
Note that the spectral measure of $A_\alpha$, a probability measure $\mu$ satisfying
\[
	\langle \mu, x^k \rangle = A_\alpha^k(1,1),\quad k = 0,1,\dots,
\]
is unique because
\[
	\sum_{j = 1}^\infty \frac{1}{\sqrt{\alpha + j}} = \infty.
\]
Also, it is clear that 
\[
	\langle \bar \mu_\alpha, x^k\rangle = \E[\langle \mu_\alpha, x^k\rangle ], \quad k = 0,1,\dots, 
\]
because $\E[\langle \mu_\alpha, |x|^k \rangle] < \infty$ for all $k=0,1,\dots.$
Therefore, our task is now to show that for all $k=0,1,\dots,$ 
\begin{equation}\label{moments-of-mu-bar}
	\langle \bar \mu_\alpha, x^k\rangle = A_\alpha^k(1,1).
\end{equation}

We consider the case of even $k$ first. For any fixed $j$, all moments of the $\chitilde_{(N-j)2\betah}$ distribution converge to those of the $\chitilde_{2\alpha}$ distribution as $N \to \infty$ with $\betah = \alpha/N$. Thus for fixed $p$, as $N \to \infty$ with $\betah = \alpha/N$,
\[
	m_{p}(N, \betah)=\E[T_N(2\betah)^{2p}(1,1)] \to \E[J_\alpha^{2p} (1,1)] = \E[\langle  \mu_\alpha, x^{2p}\rangle] .
\]
Consequently, for even $k$, namely, $k=2p$,
\[
	\langle \bar \mu_\alpha, x^{k}\rangle 
= A_\alpha^{k} (1,1),
\]
by taking into account Proposition~\ref{thm:limit-of-m}.

For odd $k$, both sides of the equation~\eqref{moments-of-mu-bar} are zeros. Indeed, $A^k_\alpha(1,1)=0$ when $k$ is odd because the diagonal of $A_\alpha$ is zero. Also all odd moments of $\bar \mu_\alpha$ are vanishing, 
\[
	\langle \bar \mu_\alpha, x^{2p+1}\rangle = \E[\langle  \mu_\alpha, x^{2p+1}\rangle] = 0,
\] 
because the expectation of odd moments of any diagonal element of $J_\alpha$ are zero. The proof is completed.
\end{proof}



\subsection{Moments of the spectral measure of $A_\alpha$}

Recall that 
\[
	u_n(\alpha) = A_{\alpha}^{2n} (1,1), n=0,1,\dots.
\]

\begin{prop}\label{lem:self-recurrence-relation}
\begin{itemize}
\item[\rm (i)]
	$u_n(\alpha)$ is a polynomial of degree $n$ in $\alpha$ and satisfies the following relations
	\begin{equation}\label{usual-relation}
		\begin{cases}
		u_n(\alpha) = (\alpha + 1) \sum_{i = 0}^{n - 1} u_i(\alpha +1) u_{n-1-i}(\alpha), \quad n \ge 1, \\
		u_0(\alpha) = 1.
		\end{cases}
	\end{equation}
\item[\rm (ii)] $\{u_n(\alpha)\}_{n = 0}^\infty$ also satisfies the following relations
	\begin{equation}\label{recursive-relation-un}
		\begin{cases}
		u_n(\alpha) = (2n -1)u_{n-1}(\alpha) + \alpha \sum_{i = 0}^{n - 1} u_i(\alpha) u_{n - 1 - i}(\alpha), \quad n \ge 1, \\
		u_0(\alpha) = 1.
		\end{cases}
	\end{equation}
\end{itemize}
\end{prop}
\begin{remark}
	The sequences $\{u_n(\alpha)\}_{n \ge 0}$,  for $\alpha = 1$ and $\alpha = 2$, are the sequences A000698 and A167872 in the On-line Encyclopedia of Integer Sequences \cite{OEIS}, respectively. Relations \eqref{usual-relation} and \eqref{recursive-relation-un} as well as many interesting properties for those sequences can be found in the above reference. In the proof below, we give another explanation of $u_n(\alpha)$ as the total sum of weighted Dyck paths of length $2n$.
\end{remark}

\begin{proof}

In this proof, for convenience, let the index of the matrix $A_\alpha$ start from $0$. Since the diagonal of $A_\alpha$ is zero, it follows that 
\[
	A_{\alpha}^{2n}(0,0) = \sum_{ \{ i_0, i_1, \dots, i_{2n} \} \in \mathfrak D_{2n}} \prod_{j = 0}^{2n-1} A_\alpha(i_j, i_{j+1}) ,
\]
where $\mathfrak D_{2n}$ denotes the set of indices  $\{i_0, i_1, \dots, i_{2n}\}$ satisfying that 
\begin{align*}
	&i_0 = 0, i_{2n} = 0, i_j \ge 0,\\
	&|i_{j+1} - i_j| = 1, j=0,1,\dots,2n-1.
\end{align*}
Each element in $\mathfrak D_{2n}$ corresponds to a path of length $2n$ consisting of rise steps or rises and fall steps or falls which starts at $(0,0)$ and ends at $(2n,0)$, and stays above the $x$-axis, called a Dyck path. We also use $\mathfrak D_{2n}$ to denote the set of all Dyck paths of length $2n$.

A Dyck path $p$ is assigned a weight $w(p)$ as follows. We assign a weight $(\alpha + k+1)$ for each rise step  from level $k$ to $k + 1$, and the weight $w(p)$ is the product of all those weights. Then 
\[
	u_n(\alpha) = A_{\alpha}^{2n}(0,0) = \sum_{p \in \mathfrak D_{2n}} w(p).
\]
\begin{figure}[ht]
    \centering
    \includegraphics[width=0.8\textwidth]{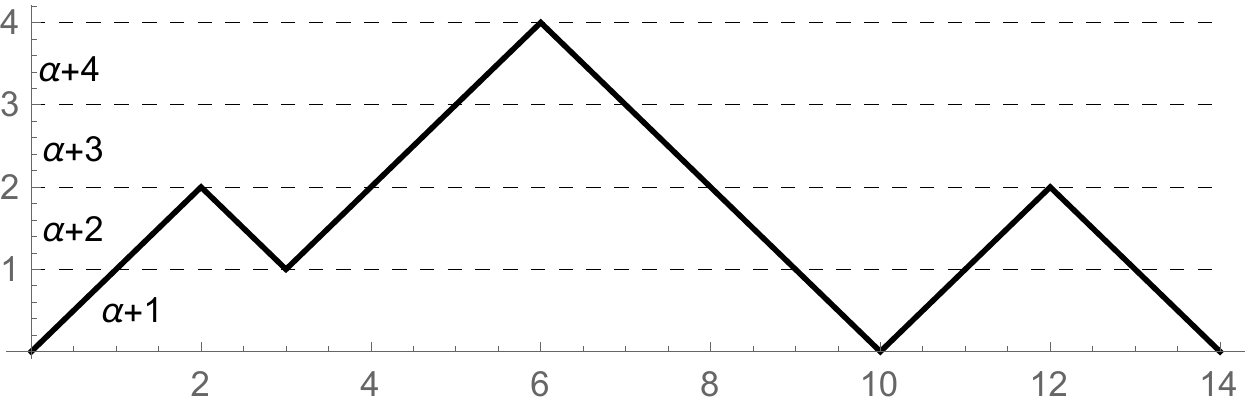}
    \caption{A Dyck path $p$ with weight $w(p) = (\alpha +1)^2(\alpha+2)^3(\alpha+3)(\alpha+4)$.}
\end{figure}

Let $\mathfrak D_{2n}^*$ be the set of all Dyck paths of length $2n$ which do not meet the $x$-axis except the starting and the ending points. Let 
\[
	v_n(\alpha) = \sum_{p \in \mathfrak D^*_{2n} }w(p).
\] 
Since each Dyck path  $p=(i_0,i_1,\dots,i_{2n-1}, i_{2n})\in\mathfrak D_{2n}^*$ is one-to-one correspondence with a Dyck path $q=(i_1-1, i_2 - 1, \dots, i_{2n-1}-1)$ of length $2(n-1)$, it follows that
\[
	v_n(\alpha) = (\alpha + 1)u_{n - 1}(\alpha + 1).
\] 
Moreover, let $2i$ be the first time that the Dyck path $p$ meets the $x$-axis. Then either $i = n$ or the Dyck path $p$ is the concatenation of a Dyck path  in $\mathfrak D_{2i}^*, (1 \le i < n),$ and another Dyck path of length $2(n - i)$. Thus,
\begin{align*}
	u_n(\alpha) &= v_n(\alpha) + \sum_{i = 1}^{n-1} v_i(\alpha) u_{n - i}(\alpha)\\
	&=(\alpha + 1) u_{n - 1}(\alpha + 1) + \sum_{i = 1}^{n - 1} (\alpha + 1) u_{i-1}(\alpha + 1) u_{n - i} (\alpha) \\
	&= (\alpha + 1) \sum_{i = 0}^{n - 1} u_i(\alpha + 1) u_{n - 1 - i}(\alpha).
\end{align*}
The proof of (i) is complete. We will prove the second statement after the  next lemma.
\end{proof}

\begin{lem}\label{lem:relation-2step}
	Let $\alpha \ge 0 $ be fixed. Let $\{a_n\}$ be a sequence defined recursively by
	\begin{equation}\label{eqdefine-a}
		\begin{cases}
		a_n = (2n - 1) a_{n - 1} + \alpha \sum_{i = 0}^{n - 1} a_i a_{n-1-i}, \quad  n \ge 1,\\
		a_0 = 1.
		\end{cases}
	\end{equation}
	Let $\{b_n\}$ be a sequence defined by the following relations $b_0 = 1,$
	\begin{equation}\label{eqdefine-ab}
		a_n = (\alpha + 1) \sum_{i = 0}^{n - 1} b_i a_{n - 1 - i}, \quad n\ge 1.
	\end{equation}
	Then $\{b_n\}$ satisfies an analogous recursive relation as $\{a_n\}$,
	\begin{equation}\label{eqdefine-b}
		\begin{cases}
		b_n = (2n - 1)b_{n - 1} + (\alpha + 1) \sum_{i = 0}^{n - 1} b_i b_{n-1-i},\quad n \ge 1,\\
		b_0=1.
		\end{cases}
	\end{equation}
\end{lem}
\begin{proof}
	Consider the field of formal Laurent series over $\R$, denoted by $\R((X))$, 
	\[
		\R((X)) = \left\{f(X) = \sum_{n \in \Z} c_n X^n : c_n \in \R, c_n = 0 \text{ for } n < n_0\right\}.
	\]
	The addition is defined as usual and the multiplication is well defined as 
	\[
		f(X) g(X) = \sum_{n \in \Z} \left( \sum_{i \in \Z} c_i d_{n - i} \right)X^n,
	\]
	for $f(X) = \sum c_n X^n, g(X) = \sum d_n X^n \in \R((X))$. The quotient $f(X) / g(X)$ is understood as $f(X)g(X)^{-1}$ for $g(X) \neq 0$. The formal derivative is also defined as 
	\[
		f'(X) = \sum_{n \in \Z} c_n n X^{n - 1} \in \R((X)).
	\]
	
	Now let 
	\[
		f(X) = \sum_{n = 0}^\infty a_n X^n,\quad  g(X) = \sum_{n = 0}^\infty b_n X^n.
	\]
	It is straightforward to show that the recursive relation \eqref{eqdefine-a} is equivalent to the following equation  
	\[
		f(X) - 1 = 2 X^2 f'(X) + Xf(X) + \alpha  X f^2(X).
	\]
	In addition, the relation \eqref{eqdefine-ab} leads to
	\[
		g(X) = \frac{f(X) - 1}{(\alpha + 1) X f(X)}.
	\]
	Finally, we can easily check that $g(X)$ satisfies 
	\[
		g(X) - 1 = 2X^2 g'(X) + X g(X) + (\alpha + 1) X g^2(X),
	\]
	which is equivalent to the recursive relation \eqref{eqdefine-b}.
	The proof is complete.
\end{proof}

\begin{proof}[Proof of Proposition~\rm{\ref{lem:self-recurrence-relation}(ii)}]
When $\alpha = 0$, it is well known that $u_n(0)$ is the $2n$th moment of the standard Gaussian distribution, and is given by 
\[
	u_n(0) = (2n -1) !!.
\]
Consequently, the conditions in Lemma~\ref{lem:relation-2step} are satisfied for $a_n = u_n(0), b_n= u_n(1)$ and $\alpha = 0$. It follows that the recursive relation \eqref{recursive-relation-un} then holds for $\alpha = 1$. Continue this way, it follows that the recursive relation \eqref{recursive-relation-un} holds for any $\alpha \in \N$. We conclude that it holds for all $\alpha$ because of the fact that $\{u_n(\alpha)\}$ is a polynomial of degree $n$ in $\alpha$. The proof is complete.
\end{proof}

\subsection{Explicit formula for the spectral measure of $A_\alpha$, proof of Theorem~\ref{thm:explicit-formula-for-mu-bar}(ii)}

In this section, by using the method of Martin and Kearney \cite{Martin2010}, we derive the explicit formula for the mean spectral measure $\bar \mu_\alpha$ from the relation \eqref{recursive-relation-un}, 
\[
	\begin{cases}
	u_n(\alpha) = (2n -1)u_{n-1}(\alpha) + \alpha \sum_{i = 0}^{n - 1} u_i(\alpha) u_{n - 1 - i}(\alpha), \quad n \ge 1,\\
	u_0(\alpha) = 1.
	\end{cases}
\]
Recall that $u_n(\alpha) = \langle \bar \mu_\alpha, x^{2n}\rangle$ and $\bar\mu_\alpha$ is a symmetric probability measure.

Let us extract here the main result of \cite{Martin2010}. The problem is to find a function $\nu$ for which 
\[
	\int_0^\infty x^{n-1} \nu(x) dx = u_n, \quad n=1,2,\dots,
\]
where the sequence $\{u_n\}$ is given by a general self-convolutive recurrence 
\begin{equation}\label{Martin}
	\begin{cases}
	u_n = (\alpha_1 n + \alpha_2) u_{n - 1} + \alpha_3 \sum_{i = 1}^{n - 1}u_i u_{n - i}, \quad n\ge 2, \\
	u_1 = 1,
	\end{cases}
\end{equation}
$\alpha_1, \alpha_2$ and $\alpha_3$ being constants. Then the solution is given by (Eq.~(13)--Eq.~(16) in \cite{Martin2010}),
\[
	\nu(x) = \frac{k(kx)^{-b} e^{-kx}}{\Gamma(a+1) \Gamma(a-b+1)} \frac{1}{U_R(kx)^2 + U_I(kx)^2},
\]
where, 
\begin{align*}
	U_R (x) &= e^{-x} \bigg( \frac{\Gamma(1 - b)}{\Gamma(a - b + 1)} {}_1F_1 (b - a;b;x)   \\
		&\hspace{3cm}- (\cos \pi b)\frac{\Gamma(b - 1)}{\Gamma(a)}x^{1 - b} {}_1F_1 (1 - a;2 - b;x) \bigg), \\
	U_I(x) &= (\sin \pi b) e^{-x}  \frac{\Gamma(b - 1)}{\Gamma(a)} x^{1 - b} {}_1F_1 (1 - a;2 - b;x),
\end{align*}
and $k = 1/\alpha_1, a= \alpha_3/\alpha_1, b = -1- \alpha_2/\alpha_1$, provided $\alpha_1 \ne 0$.
Here ${}_1 F_1(a;b;z)$ is the Kummer function.

The sequence $\{u_n(\alpha)\}_{n \ge 0}$ is a particular case of the self-convolutive recurrence~\eqref{Martin} with parameters $\alpha_1 = 2, \alpha_2 = -3$ and $\alpha_3 = \alpha$. Note that our sequence $\{u_n(\alpha)\}$ starts from $n = 0$, and thus $\alpha_2 = -3$. By direct calculation, we get $k = 1/2, a = \alpha/2,$ and  $b = 1/2$. Therefore, the function $\nu_\alpha(x)$ for which $u_n(\alpha) = \int_0^\infty x^n d\nu_\alpha(x) dx, n=0,1,\dots,$ is given by 
\[
	\nu_{\alpha}(x) = \frac{1}{\sqrt{2}\Gamma(\frac{\alpha}{2} + 1)\Gamma(\frac{\alpha}{2} + \frac12) } \frac{1}{\sqrt{x}} e^{-\frac x 2}  \frac{1}{U_R(x/2)^2 + U_I(x/2)^2}, \quad x>0,
\]
where
\begin{align}
	U_R (x) &= e^{-x}  \frac{\Gamma(\frac 12)}{\Gamma(\frac{\alpha}{2}+ \frac{1}{2})} {}_1F_1 (\frac{1}{2} - \frac{\alpha}{2};\frac{1}{2};x)  , \label{UR}\\
	U_I(x) &=  e^{-x}  \frac{\Gamma(-\frac{1}{2})}{\Gamma(\frac{\alpha}{2})} x^{1/2} {}_1F_1 (1 - \frac{\alpha}{2};\frac{3}{2};x). \label{UI}
\end{align}
It is clear that $\nu_\alpha(x) > 0$ for any $x>0$. Now it is easy to check that the function  $\bar \mu_\alpha (y)$ defined by
\[
	\bar \mu_\alpha (y) = |y| \nu_\alpha(y^2), \quad y \in \R,
\]
satisfies the following relations  
\[
	\int_{\R} y^{2n + 1}  \bar \mu_\alpha (y)dy = 0, \quad \int_{\R} y^{2n}\bar \mu_\alpha (y) dy = u_n(\alpha), \quad n = 0,1,\dots, 
\]
In other words, $\bar \mu_\alpha (y)$ is the density of the mean spectral measure $\bar \mu_\alpha$ with respect to the Lebesgue measure.


We are now in a position to simplify the explicit formula of $\bar \mu_\alpha$. 
Let 
\begin{align*}
	V_R(y) &=  \left( \frac{\Gamma(\frac{\alpha}{2} + 1)\Gamma(\frac{\alpha}{2} + \frac12)}{\Gamma(\frac 12)} \right)^{1/2} U_R(y^2/2),\\
	&=  2^{-\frac\alpha 2}\Gamma(\alpha + 1)^{\frac 12}  \frac{\Gamma(\frac 12)}{\Gamma(\frac{\alpha}{2}+ \frac{1}{2})}e^{-\frac{y^2}{2}}  {}_1F_1 (\frac{1}{2} - \frac{\alpha}{2};\frac{1}{2};\frac{y^2}{2}),\\
	V_I(y) &=- \left( \frac{\Gamma(\frac{\alpha}{2} + 1)\Gamma(\frac{\alpha}{2} + \frac12)}{\Gamma(\frac 12)} \right)^{1/2} U_I(y^2/2)\\
	&=-2^{-\frac\alpha 2 -\frac12}\Gamma(\alpha + 1)^{\frac 12}  \frac{\Gamma(-\frac{1}{2})}{\Gamma(\frac{\alpha}{2})} {y}e^{-\frac{y^2}{2}} {}_1F_1 (1 - \frac{\alpha}{2};\frac{3}{2};\frac{y^2}{2}).
\end{align*}
Here, in the above expressions, we have used the following relation for Gamma function
\begin{equation}\label{relation-Gamma-function}
	\frac{\Gamma (\frac\alpha2 + \frac12)\Gamma (\frac\alpha2 + 1)}{\Gamma(\frac12)} = 2^{-\alpha}\Gamma(\alpha + 1).
\end{equation}
Then $\bar \mu_\alpha (y)$ can be written as 
\[
	\bar \mu_\alpha (y) = \frac{e^{-\frac{y^2}2} }{\sqrt{2\pi}} \frac{1}{V_R(y)^2 + V_I(y)^2}.
\]

Next, we will show that $V_R(y)$ and $V_I(y)$ are the Fourier cosine transform and Fourier sine transform of 
\[
	f_\alpha(t) = \pi \sqrt{\frac{\alpha}{\Gamma(\alpha)}} t^{\alpha - 1} \frac{e^{-\frac {t^2}{2}}}{\sqrt{2\pi}},
\]
respectively.
Let us now give definitions of Fourier transforms. The Fourier transform of a function $f\colon \R \to \C$ is defined to be
\[
	\F(f)(y) = \frac 1{\sqrt{2\pi}} \int_{-\infty}^\infty f(t) e^{i y t} dt,\quad (y \in \R),
\]
and the Fourier cosine transform, the Fourier sine transform are defined to be 
\begin{align*}
	\F_c(f)(y) = \sqrt{\frac 2\pi} \int_0^\infty f(t) \cos(yt) dt, \quad (y > 0),\\
	\F_s(f)(y) = \sqrt{\frac 2\pi} \int_0^\infty f(t) \sin(yt) dt,\quad (y > 0),
\end{align*} 
respectively. Then those transforms are related as follows
\[
	\begin{cases}
		\F(f)(y) = \F_c(f)(y),\quad ( y \ge 0), &\text{if $f(t)$ is even,}\\
		\F(f)(y) = i\F_s(f) (y),\quad ( y\ge 0), &\text{if $f(t)$ is odd.}
	\end{cases}
\]

For $\alpha > 0$, we have (cf.~Formula 3.952(8) in \cite{Gradshteyn-Ryzhik2007})
\begin{equation*}
	\F_c(t^{\alpha - 1} e^{-\frac{t^2}{2}}) = \frac{2^{\frac\alpha2 - \frac 12} \Gamma(\frac\alpha 2)}{\sqrt{\pi}} e^{-\frac{y^2}{2}} {}_1F_1(\frac 12 - \frac{\alpha}{2}; \frac 12; \frac {y^2}{2}).
\end{equation*}
Then by some simple calculations, we arrive at the following relation
\[
	V_R(y) = \F_c(f_\alpha(t))(y), \quad y\ge 0.
\]
Similarly, 
\[
	V_I(y) =  \F_s(f_\alpha(t))(y), \quad y\ge 0,
\]
by using Formula~3.952(7) in \cite{Gradshteyn-Ryzhik2007}, 
\begin{equation*}
	\F_s(t^{\alpha - 1} e^{-\frac{t^2}{2}}) = \frac{2^{\frac\alpha2} \Gamma(\frac\alpha 2+\frac12)}{\sqrt{\pi}} y e^{-\frac{y^2}{2}} {}_1F_1(1  - \frac{\alpha}{2}; \frac 32; \frac{y^2}{2}).
\end{equation*}

	By definitions, $V_R(y)$ is an even function and $V_I(y)$ is an odd function. Thus the following expression holds for all $y \in \R$,
\begin{align*}
	V_R(y) + i V_I(y) &= \sqrt{\frac 2\pi} \int_0^\infty f_\alpha(t) (\cos(yt) + i\sin(yt) dt \\
	&= \sqrt{\frac 2\pi} \int_0^\infty f_\alpha(t) e^{i yt}dt =: \hat f_\alpha(y).
\end{align*}
Consequently, 
\[
	V_R(y)^2 + V_I(y)^2 = |\hat f_\alpha(y)|^2,
\]
which completes the proof of Theorem~\ref{thm:explicit-formula-for-mu-bar}(ii).



We plot the graph of the density $\bar \mu_\alpha(y)$ for several values $\alpha$ as in the following figure by using Mathematica. It follows from the Jacobi matrix form that the spectral measure of $\frac{1}{\sqrt{\alpha}}A_\alpha $ converges weakly to the semicircle law as $\alpha$ tends to infinity. Note that the semicircle law, the probability measure supported on $[-2,2]$ with the density
\[
	\frac{1}{2\pi} \sqrt{4 - x^2}, (-2 \le x \le 2),
\]
is the spectral measure of the following Jacobi matrix 
\[
	\begin{pmatrix}
		0		&1	\\
		1		&0		&1		\\
				&1		&0		&1\\
									&&\ddots		&\ddots		&\ddots
	\end{pmatrix}.
\]
\begin{figure}[ht]
    \centering
    \includegraphics[width=0.8\textwidth]{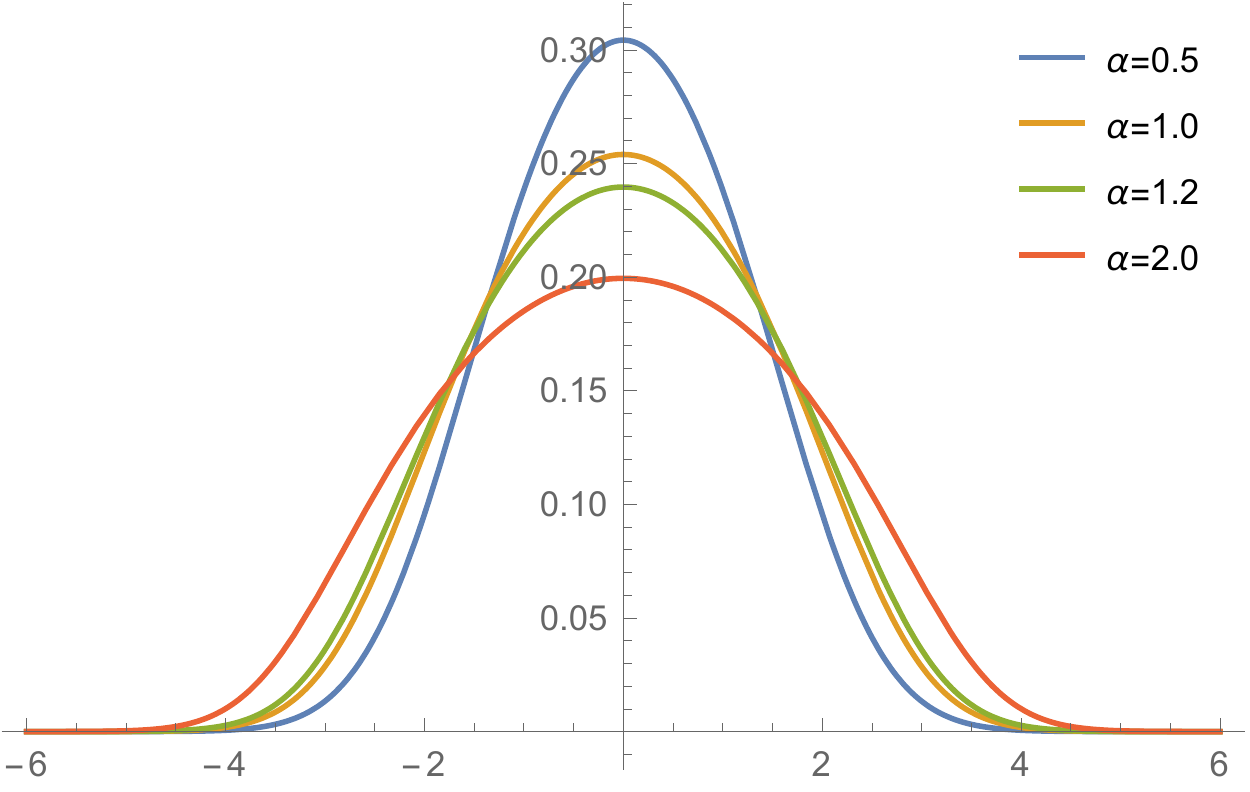}
    \caption{The density $\bar\mu_\alpha(y)$ for several values $\alpha$.}
    \label{fig:density}
\end{figure}

\begin{remark}
When $\alpha$ in a positive integer number, we can give even more explicit expressions for $V_R(y)$ and $V_I(y)$.
\begin{itemize}
\item[(i)] $ \alpha = 2n, n \in \N$. In this case, $f_\alpha(t)$ is an odd function. Therefore 
\[
	V_I(y) = \F_s(f_\alpha(t)) = -i \F(f_\alpha(t)).
\] 
Note that 
\[
	\F(e^{-\frac{t^2}{2}}) = e^{-\frac{y^2}{2}}.
\]
Therefore, for integer $\alpha \ge 1$,
\[
	\F(t^{\alpha - 1} e^{-\frac{t^2}{2}}) = (i)^{\alpha - 1} \frac{d^{\alpha -1}}{dy^{\alpha - 1}}(e^{-\frac{y^2}{2}}).
\]
Consequently,
\begin{align*}
	V_I(y)  &= -i^{\alpha} \pi \sqrt{\frac{\alpha}{\Gamma(\alpha)}}\frac{d^{\alpha -1}}{dy^{\alpha - 1}}(e^{-\frac{y^2}{2}})  \frac{1}{\sqrt{2\pi}} \\
	&= -i^{\alpha} \pi \sqrt{\frac{\alpha}{\Gamma(\alpha)}} e^{\frac{y^2}{2}}\frac{d^{\alpha -1}}{dy^{\alpha - 1}}(e^{-\frac{y^2}{2}})  \frac{e^{-\frac{y^2}{2}}}{\sqrt{2\pi}} \\
	&=- i^{\alpha} \pi \sqrt{\frac{\alpha}{\Gamma(\alpha)}} He_{\alpha-1} \frac{e^{-\frac{y^2}{2}}}{\sqrt{2\pi}}.
\end{align*}
Here $He_m$ denotes probabilists' Hermite polynomials.

\item[(ii)] $\alpha = 2n + 1$. This case is very similar. Since $f_\alpha(t)$ is an even function, it follows that 
\[
	V_R(y) = \F_c(f_\alpha(t)) (y) = \F(f_\alpha(t)) = i^{\alpha - 1} \pi \sqrt{\frac{\alpha}{\Gamma(\alpha)}} He_{\alpha-1} \frac{e^{-\frac{y^2}{2}}}{\sqrt{2\pi}}.
\]

\end{itemize}
\end{remark}

{}

\hfill
\begin{tabular}{l}
Trinh Khanh Duy \\
Institute of Mathematics for Industry \\
Kyushu University\\
Fukuoka 819-0395, Japan \\
e-mail: trinh@imi.kyushu-u.ac.jp \\
\\

Tomoyuki Shirai \\
Institute of Mathematics for Industry \\
Kyushu University\\
Fukuoka 819-0395, Japan \\
e-mail: shirai@imi.kyushu-u.ac.jp \\

\end{tabular}

\end{document}